\documentclass[12pt]{amsart}

\usepackage[latin1]{inputenc}					
\usepackage{amsmath,amsfonts,amssymb,amsthm}	

\usepackage{geometry}		
\usepackage{mathtools}		
\usepackage{hyperref}		
\usepackage{setspace}		
\usepackage{cite}			

\hypersetup{colorlinks=true,citecolor=blue,linkcolor=blue,urlcolor=magenta}

\numberwithin{equation}{section}			
\newtheorem{theoremintro}{Theorem}

\newtheorem{theorem}{Theorem}[section]
\newtheorem{proposition}[theorem]{Proposition}
\newtheorem{lemma}[theorem]{Lemma}
\newtheorem{corollary}[theorem]{Corollary}

\newtheorem*{theorem*}{Theorem}
\newtheorem*{proposition*}{Proposition}
\newtheorem*{lemma*}{Lemma}
\newtheorem*{corollary*}{Corollary}

\newtheorem*{conjecture*}{Conjecture}
\newtheorem*{question*}{Question}
\newtheorem*{remark*}{Remark}
\newtheorem*{definition*}{Definition}

\renewcommand{\leq}{\leqslant}	
\renewcommand{\geq}{\geqslant}	

\newcommand{\N}{\mathbb{N}}		
\newcommand{\Z}{\mathbb{Z}}		
\newcommand{\R}{\mathbb{R}}		
\newcommand{\C}{\mathbb{C}}		
\newcommand{\T}{\mathbb{T}}		

\newcommand{\eps}{\varepsilon}			

\newcommand{\wt}{\widetilde}		
\newcommand{\wh}{\widehat}			

\newcommand{\E}{\mathbb{E}}			

\newcommand{\bfz}{\mathbf{z}}
\newcommand{\bfn}{\mathbf{n}}
\newcommand{\bfM}{\boldsymbol{M}}
\newcommand{\bfgamma}{\boldsymbol{\gamma}}

\begin{document}

\title{Logarithmic bounds for translation-invariant equations in squares}

\author{Kevin Henriot}

\date{}

\begin{abstract}
	We show that the equation
	$\lambda_1 n_1^2 + \dotsb + \lambda_s n_s^2 = 0$
	admits non-trivial solutions in any subset of $[N]$
	of density $(\log N)^{-c_s}$, provided that $s \geq 7$
	and the coefficients $\lambda_i \in \Z \smallsetminus \{0\}$
	sum to zero and satisfy certain sign conditions.
	This improves upon previous known density bounds
	of the form $(\log\log N)^{-c}$.
\end{abstract}

\maketitle

\section{Introduction}
\label{sec:intro}

We are interested in quantitative quadratic analogues
of Roth's theorem~\cite{roth:roth}, which asserts
the existence of non-trivial
three-term arithmetic progressions
in subsets of $[N]$ of density at least $(\log\log N)^{-1}$.
It was shown in the late 1980s by Heath-Brown~\cite{heath-brown:roth}
and Szemerédi~\cite{szemeredi:roth} that better bounds of the
form $(\log N)^{-c}$ hold for this theorem.
A well-studied nonlinear analogue is Sarközy's theorem~\cite{sarkozy:sark},
which establishes the existence of square differences
in sets of density $(\log N)^{-1/3+o(1)}$,
however it is less directly relevant to our work, since we
consider equations with all variables lying in a
dense set.

Roth's method and subsequent improvements generalize
readily to all translation-invariant equations,
that is to say all equations of the form
\begin{align}
\label{eq:intro:tslinveq}
	\lambda_1 x_1 + \dots + \lambda_s x_s = 0,
\end{align}
where $\lambda_i \in \Z \smallsetminus \{0\}$,
$\lambda_1 + \dots + \lambda_s = 0$
and $x_1,\dots,x_s$ are integer variables.
From an additive number-theoretic point of view,
it is natural to ask whether~\eqref{eq:intro:tslinveq}
is solvable in a dense subset of an arithmetic set instead,
such as that of the primes or the squares.
In the case of primes,
the analogue of Roth's theorem was shown to hold
by Green~\cite{green:rothprimes}, who for that purpose developped
the now-standard transference principle,
and used restriction estimates for primes in arithmetic progressions
akin to those of Bourgain~\cite{bourgain:restrsquares} for primes.
By contrast, the case of squares can be handled
directly, although rather uneconomically,
by invoking the formidable
bounds of Gowers~\cite{gowers:sz}
for Szemerédi's theorem.
Indeed, let $x_i = n_i^2$ in~\eqref{eq:intro:tslinveq} 
and consider the larger system
\begin{align}
\label{eq:intro:quadsyst}
	\begin{split}
		\lambda_1 n_1 + \dotsb + \lambda_s n_s &= 0, \\
		\lambda_1 n_1^2 + \dotsb + \lambda_s n_s^2 &= 0.	
	\end{split}
\end{align}
in variables $n_1,\dots,n_s$, which has the crucial
property of being invariant under translation and dilation.
We call a solution $(n_1,\dots,n_s)$ to the above non-trivial
when the $n_i$ are all distinct.
Assuming the existence of a non-trivial integer solution
$(m_1,\dots,m_s)$ to~\eqref{eq:intro:quadsyst},
Gowers' argument~\cite{gowers:sz} enables one to locate
a pattern of the form $(x + d m_1,\dots,x + d m_s)$
with $x,d \in \N$ in any subset of $[N]$ of density
at least $C_s(\log\log N)^{-c_s}$ with $c_s = 2^{-2^{s+9}}$,
and this pattern again satisfies~\eqref{eq:intro:quadsyst}.

However, for large $s$,
one can expect the system~\eqref{eq:intro:quadsyst}
to be governed by Fourier analysis rather than higher-order uniformity,
and the quantitative bounds to be superior.
The work of Smith~\cite{smith:diagonal}, who initiated
the study of the system~\eqref{eq:intro:quadsyst}
in dense sets, and that of Keil~\cite{keil:diagonal}
progressively confirmed both of these expectations.
These two authors pursued a density increment strategy
made possible by the translation and dilation invariance of the system,
where at each step the set either contains the expected number of solutions
given by the classical circle method, or correlates
with a small arithmetic progression.
Smith relied on quadratic Fourier analysis
and on the Hardy-Littlewood circle method to
show that the system~\eqref{eq:intro:quadsyst} admits
non-trivial solutions in sets of density
as low as $(\log\log N)^{- 8 \cdot 10^{-7}}$, for $s \geq 9$
and certain conditions on the coefficients.
Keil later simplified Smith's approach to use only linear Fourier analysis,
and reduced the necessary number of variables to $7$,
respectively by using appropriately the technique of linearization of a quadratic
from Gowers' work~\cite{gowers:sz}
and by incorporating (as well as giving a new proof of)
certain restriction estimates of Bourgain~\cite{bourgain:restrlattice} for lattice sets.

\begin{theoremintro}[Keil]
\label{thm:intro:keilthm}
	Assume that $s \geq 7$ and 
	$\lambda_1,\dots,\lambda_s \in \Z \smallsetminus \{0\}$
	are such that ${\sum_{i=1}^s \lambda_i = 0}$, and that
	at least two of these coefficients are positive
	and at least two are negative.
	There exists a positive constant $C$ depending 
	on $s$ and $\lambda_1,\dots,\lambda_s$
	such that the following holds.
	Suppose that $A$ is a subset of $[N]$ of density
	\begin{align*}
		\delta \geq C(\log\log N)^{-1/15}.
	\end{align*}
	Then $A$ contains a non-trivial solution to the
	system of equations~\eqref{eq:intro:quadsyst}.
\end{theoremintro}

The translation-invariance and sign conditions 
on the coefficients, which in particular force $s \geq 4$,
are necessary for the theorem to hold,
as explained by Keil~\cite{keil:diagonal}.
The constraint $s \geq 7$, on the other hand, is due to limitations
in the classical circle method.
The drawback of Theorem~\ref{thm:intro:keilthm} is that
the density bound is still doubly logarithmic,
in contrast with the linear case where logarithmic bounds are available.
Confirming an expectation of Smith~\cite[p.~276]{smith:diagonal},
we remove this discrepancy.

\begin{theoremintro}
\label{thm:intro:mainthm}
	Assume that $s \geq 7$ and 
	$\lambda_1,\dots,\lambda_s \in \Z \smallsetminus \{0\}$
	are as in Theorem~\ref{thm:intro:keilthm}.
	There exists a positive constant $c$
	depending on $s$ and $\lambda_1,\dotsc,\lambda_s$
	such that the following holds.
	Suppose that $A$ is a subset of $[N]$ of density
	\begin{align*}
		\delta \geq 2(\log N)^{-c}.
	\end{align*}
	Then $A$ contains a non-trivial solution
	to the system of equations~\eqref{eq:intro:quadsyst}.
\end{theoremintro}

Our main input is to adapt to the problem at hand
the energy-increment strategy
of Heath-Brown~\cite{heath-brown:roth}
and Szemerédi~\cite{szemeredi:roth},
by which one collects several large frequencies of the Fourier transform
to obtain a more efficient density increment.
In fact we use the framework from Green's exposition~\cite{green:roth} of this technique,
where the discrete Fourier transform is used to simplify combinatorial arguments.
Our proof also relies on the circle method analysis of Smith~\cite{smith:diagonal}
and Keil~\cite{keil:diagonal}, to which we make no further contribution,
and again on the estimates of Bourgain~\cite{bourgain:restrlattice}
for exponential sums of the form ${\sum_n a_n e(\alpha n + \beta n^2)}$.
In our situation however, we now fully exploit Bourgain's $L^p \,\text{--}\, L^2$ estimate,
and we could not content ourselves with an $L^p \,\text{--}\, L^\infty$ bound.
A new ingredient of our approach is the use of simultaneous linearization
of quadratics, a technique developped by Green and Tao~\cite{GT:quantsz}
in the context of deriving efficient bounds for Szemerédi's theorem for
progressions of length four.
Our setting is different than in the higher-order situation and
we work only with arithmetic progressions, but we have to be careful
with how this process interacts with the arithmetic energy increment strategy.

The dependency of the exponent $c$ on the coefficients
$\lambda_i$ in Theorem~\ref{thm:intro:mainthm}
is an unescapable feature of our argument, and even in the
case where $s = 7$ and $\lambda_i = \pm 1$ (say)
we do not expect it to produce numerically very efficient constants.
This seems to be an intrinsic limitation in the original
method of Heath-Brown~\cite{heath-brown:roth}
and Szemerédi~\cite{szemeredi:roth}, and obtaining competitive exponents
would likely require an adaptation of the machinery
of density increment on Bohr sets
developped by Bourgain~\cite{bourgain:roth}.
We note finally that 
Theorem~\ref{thm:intro:mainthm} would 
follow as a corollary if logarithmic bounds
for Szemerédi's theorem for progressions of 
length $s$ were to be established,
and in fact there is ongoing work in that direction
by Green and Tao for $s = 4$, but the cases $s \geq 5$
seem far from accessible at present.

In the converse direction, an argument
of Shapira~\cite{shapira:behrend}
shows that for most choices of $(\lambda_i)$, there
exist Behrend-type sets of density $e^{-c\sqrt{\log N}}$
containing no solutions to the first equation in~\eqref{eq:intro:quadsyst},
although this does not answer the question of solving
the squares equation alone.
In the special case where $s = 3$ and $(\lambda_i) = (1,1,-2)$,
which is currently out of reach of analytic methods,
a construction of Gyarmati and Ruzsa~\cite{GR:squaresnoAP} gives a set of density
$(\log\log N)^{-1/2}$ without solutions to the second equation
in~\eqref{eq:intro:quadsyst}.
For small values of $s$, we do not have a good heuristic
for what the true density bound in Theorem~\ref{thm:intro:mainthm} should be,
but for large values it is tempting to conjecture Behrend-type bounds,
in light of recent developments in Roth's theorem
in many variables~\cite{schoenshkredov:rothmany,schoensisask:rothmany}.

\textbf{Acknowledgements.} We thank Farzad Aryan for
interesting discussions on the problem studied in this paper.

\textbf{Funding.} This work was partially supported
by the ANR Caesar {ANR-12-BS01-0011}.

\section{Notation}
\label{sec:notation}

In this preliminary section we introduce our basic notation
and our normalization conventions.
We write $\T = \R/\Z$ for the torus, equipped with
the distance $\| \cdot \| = d( \,\cdot , \Z)$.
We write $X \sim Y$ when $X,Y$ are positive reals such that
$Y \leq X \leq 2Y$.
We also let $[N]$ denote the interval $\{1,\dotsc,N\}$
when $N$ is an integer.
We let $c$ and $C$ denote positive constants whose
values may change from line to line, and when
we want to temporarily fix those values we
use subscripts $c_1,c_2$, and so on.

Given integers $N_1,\dots,N_d \geq 1$ and a function 
$F : \Z_{N_1} \times \dotsb \times \Z_{N_d} \rightarrow \C$,
we define the $L^p$ norm of $F$ for $p \geq 1$ by
\begin{align*}
	\| F \|_p = \textstyle 
	\big( \E_{n_1 \in \Z_{N_1}, \dots, n_d \in \Z_{N_d}} |F(n_1,\dots,n_d)|^p \big)^{1/p},
\end{align*}
For a function $f : \Z \rightarrow \C$ and an integer $N \geq 1$, we define
\begin{align*}
	\| f \|_{L^p(N)} = \big( \E_{n \in [N]}\, |f(n)|^p )^{1/p}. 
\end{align*}
We also use the standard $\ell^p$ and $L^p$ norms respectively on $\Z^d$ and on $\T^d$,
and we write $\| f \|_p = \| f \|_{\ell^p(\Z^d)}$ for functions on $\Z^d$
and $\| f \|_p = \| f \|_{L^p(\T^d)}$ for functions on $\T^d$.
We occasionally use the Fourier transform of a function 
$g : \Z^d \rightarrow \C$ with finite support, which is then defined by
$\wh{g}(\bfgamma) 
= \sum_{\bfn \in \Z^d} g(\bfn) e( \bfn \cdot \bfgamma )$
for $\bfgamma \in \T^d$.

\section{Outline and organization of the paper}
\label{sec:outline}

Our argument is modeled after the original
energy increment strategy of Heath-Brown~\cite{heath-brown:roth}
and Szemerédi~\cite{szemeredi:roth},
which we briefly recall, following the exposition of Green~\cite{green:roth}.
In the case of three-term arithmetic progressions, the counting operator 
acting on functions $f_i : \Z \rightarrow \C$ with support in $[N]$ takes the form
\begin{align*}
	M^{-2} \sum_{\substack{n_1,n_2,n_3 \in \Z : \\ n_1 + n_2 = 2n_3}}
	f_1 (n_1) f_2 (n_2) f_3 (n_3)
	= \sum_{r \in \Z_M} \wh{f}_1 (r) \wh{f}_2 (r) \wh{f}_3 (-2r),
\end{align*}
where $M = 2N$ and we have defined the Fourier transform $\wh{f} : \Z_M \rightarrow \C$ by
\begin{align}
\label{eq:outline:linearfouriertsf}
	\wh{f}(r) = \E_{n \in [M]} f(n) e\Big( \frac{nr}{M} \Big).
\end{align}
When $A$ is a subset of $[N]$ of density $\delta$ 
containing no (non-trivial) three-term arithmetic progressions,
the usual multilinear expansion process coupled with Hölder's inequality
shows that $\| \wh{f}_A \|_3 \gg \delta$,
where $f_A = 1_A - \delta 1_{[N]}$.
Via a clever ``energy pigeonholing'' argument, one can then extract a larger
restricted moment of lower order ($2$ in this case),
that is to say one may find an integer $1 \leq R \ll \delta^{-C}$ 
and distinct frequencies $r_1,\dots,r_R \in \Z_M$ such that
\begin{align}
\label{eq:outline:largerestrictedl2}
	R^c \ll \sum_{i=1}^R \bigg| \frac{\wh{f}_A}{\delta}(r_i) \bigg|^2.
\end{align}
The next step is to observe that the phases $x \mapsto e(r_i x/M)$
are simultaneously approximately constant on a progression $P$ of length
$\sim N^{c/R}$ by Dirichlet's theorem on
simultaneous linear recurrence~\cite{schmidt:diophapprox}.
This can be used in the expansion~\eqref{eq:outline:linearfouriertsf}
to show that $\wh{f}_A(r_i) = \wh{f_A \ast \mu_P}(r_i) \,+\, O(N^{-c/R})$
for all $i \in [R]$, an error term which is strong enough
to replace $f_A$ by its convolution with $\mu_P$
in~\eqref{eq:outline:largerestrictedl2} at little to no cost.
Completing the sum in~\eqref{eq:outline:largerestrictedl2}
and using Plancherel's identity, it follows that
$\| f_A \ast \mu_P \|_{L^2(N)}^2 \gg \delta R^c$.
By unfolding $f_A = 1_A - \delta 1_{[N]}$,
performing a couple regularity computations and
using an $L^\infty \,\text{--}\, L^1$ 
bound\footnote{
The argument in Green's note~\cite{green:roth}
is in fact even simpler in that it involves the
function $1_A$ instead of $f_A$, and thus does not
require the said regularity computations.
However it requires an additional Fejer-smoothing operation
in the counting step, which we could not reproduce here:
see the remarks at the end of the article.
},
one is quickly led to a density increment of the form
\begin{align*}
	\delta &\leftarrow (1 + cR^c) \cdot \delta, \\
	N &\leftarrow N^{c/R},
\end{align*}
upon passing to the smaller arithmetic progression and rescaling.
Iterating these bounds, one obtains a logarithmic density bound
in Roth's theorem, and the reason behind this success is the 
efficient lower bound in~\eqref{eq:outline:largerestrictedl2}, 
which crucially does not lose any power of~$\delta$ and even
adds a gain factor $R$.

In the quadratic case then, the counting operator
associated to the system~\eqref{eq:intro:quadsyst} is
\begin{align*}
	M^{-(s-3)} \quad \sum_{ \mathclap{ n_1,\dots,n_s \in \Z_M \,:\, \eqref{eq:intro:quadsyst} } } \quad
		f_1(n_1) \dotsc f_s(n_s)
	=  \quad \sum_{ \mathclap{\mathbf{z} \in \Z_M \times \Z_{M^2}} } \quad
	S_{f_1}(\lambda_s \mathbf{z}) \dotsc S_{f_s}(\lambda_s \mathbf{z})
\end{align*}
where $s \geq 7$ and $S_f : \Z_M \times \Z_{M^2} \rightarrow \C$
is defined at $\bfz = (x,y)$ by
\begin{align}
\label{eq:outline:quadfouriertsf}
	S_f ( \bfz ) = \E_{n \in [M]} f(n) e\Big( \frac{xn}{M} + \frac{yn^2}{M^2} \Big).
\end{align}
This last expression is best interpreted as a quadratic version
of the usual Fourier transform.
By the discrete version of a restriction estimate
of Bourgain~\cite{bourgain:restrlattice},
we know that $S_f$ has bounded moments of order $> 6$
for bounded $f$, while the $6$-th moment for $f = 1_{[N]}$ is already
unbounded due to arithmetic considerations~\cite{rogovskaya:6thmoment}.
This suggests the following modification to the previous strategy:
Starting with a set with fewer solutions to~\eqref{eq:intro:quadsyst}
than expected from the circle method,
we now extract a large restricted sum
\begin{align}
\label{eq:outline:largerestrictedlr}
	R^c \ll \sum_{i=1}^R |S_{f_A / \delta}(\mathbf{z}_i)|^r,
\end{align}
with $6 < r < s$;
this is done in Sections~\ref{sec:multlin} and~\ref{sec:energy}.
This time we simultaneously linearize the corresponding
quadratic phases $n \mapsto e( \alpha_i n + \beta_i n^2 )$,
taking our inspiration from Green and Tao~\cite{GT:quantsz} and
relying in particular on their version of 
Schmidt's result on simultaneous quadratic recurrence~\cite{schmidt:fractionalparts}.
This lets us replace $f_A$ in~\eqref{eq:outline:largerestrictedlr}
by a smoothed version $\wt{f}_A$, which corresponds roughly
to an average of convolutions of $f_A$ with
smaller progressions of size $\sim N^{c/R^3}$.
Completing the sum in~\eqref{eq:outline:largerestrictedlr},
and applying Bourgain's $\ell^r \,\text{--}\, L^2$ restriction
bound~\cite{bourgain:restrlattice} in place of Plancherel's identity,
we deduce that
\begin{align*}
	R^c \ll \| S_{\wt{f}_A / \delta} \|_r \ll \| \wt{f}_A / \delta \|_{L^2(M)},
\end{align*}
and from there the rest of the argument
proceeds much as in the linear case.
The linearization and the $L^2$ density increment steps
are the most technical parts of this work and
are found in Section~\ref{sec:lin}.
The main density increment lemma is proved in
Section~\ref{sec:pieces}, but the final 
density increment iteration is
carried out prior to all of these steps,
in Section~\ref{sec:dincr}.
Section~\ref{sec:rks} contains
technical remarks on possible simplifications
or extensions of the argument.

\section{The density increment iteration}
\label{sec:dincr}

In this section we deduce
Theorem~\ref{thm:intro:mainthm} 
from the following density increment statement, 
whose proof will occupy the rest of this article.
Although formally this is the last logical step of our argument,
it is convenient to dispense with it at the beginning of the article,
so that we may fix a scale $N$ at the outset.
The iteration process below is essentially that carried out in
Green's note~\cite{green:roth}.

\begin{proposition}
\label{thm:dincr:itlemma}
    Suppose that $s \geq 7$ and $\lambda_1,\dots,\lambda_s \in \Z \smallsetminus \{0\}$
    are as in Theorem~\ref{thm:intro:keilthm}.
    There exist positive constants $c_1,c_2,c_3,D$ depending on
    $s$ and $\lambda_1,\dotsc,\lambda_s$ such that the following holds.
    Suppose that $A$ is a subset of $[N]$ of density $\delta$
    containing no non-trivial solutions
    to~\eqref{eq:intro:quadsyst}, and that $N \geq e^{(2/\delta)^D}$.
    Then there exists an arithmetic progression $Q$ contained in $[N]$
    and an integer $R \geq 1$ such that
    \begin{align*}
        |A \cap Q| / |Q| &\geq (1 + c_1 R^{c_2}) \cdot \delta, \\
        |Q| &\geq N^{c_3/R^3}.
    \end{align*}
\end{proposition}

\textit{Proof of Theorem~\ref{thm:intro:mainthm}.}
We consider a subset $A$ of $[N]$ of density $\delta$,
and we construct iteratively a sequence of subsets
$A_i$ of intervals $[N_i]$ of density $\delta_i$, 
initializing at $(A_0,N_0) = (A,N)$.
As long as $\delta_i (\log N_i)^{1/D} \geq 2$,
we run the following algorithm:
If $A_i$ contains no non-trivial
solutions to~\eqref{eq:intro:quadsyst}, 
we let $Q_{i+1} = u_{i+1} + q_{i+1} [N_{i+1}]$ be
the arithmetic progression given by Proposition~\ref{thm:dincr:itlemma},
and we define $A_{i+1}$ by 
$A_i \cap Q_{i+1} = u_{i+1} + q_{i+1} A_{i+1}$.
Thus there exists $R_i \geq 1$ such that
\begin{align}
\label{eq:dincr:itbounds}
\begin{split}
	\delta_{i+1} &\geq (1 + c_1 R_i^{c_2}) \cdot \delta_i, \\
	N_{i+1} &\geq N_i^{c_3/R_{i}^3}.
\end{split}
\end{align}
In particular, $\delta_{i+1} \geq (1 + c_1) \cdot \delta_i$
and the algorithm carries on for a finite number of steps,
since all densities are bounded by $1$.
We now show that there exists a positive constant
$\kappa = \kappa(c_1,c_2,c_3,D) \leq 1/D$ such that
if $\delta (\log N)^\kappa \geq 2$,
then $\delta_i (\log N_i)^\kappa \geq 2$ for every $i$
for which $A_i$ is defined.
Indeed by~\eqref{eq:dincr:itbounds}, it follows that
\begin{align*}
	\frac{\delta_{i+1}(\log N_{i+1})^{\kappa}}{\delta_i (\log N_i)^{\kappa}} 
	\geq \frac{1 + c_1 R_i^{c_2}}{(R_i^3/c_3)^\kappa} \eqqcolon f(R_i).
\end{align*}
Assuming that $\kappa \leq c_2 / 6$, we clearly have $f(R) \geq 1$ 
for $R \geq C(c_1,c_2,c_3)$ large enough,
and for $R \leq C(c_1,c_2,c_3)$ we can also guarantee that $f(R) \geq 1$
by choosing $\kappa \leq c(c_1,c_2,c_3)$ small enough.

Assume now that $\delta \geq 2(\log N)^{-\kappa}$,
so that the condition $\delta_i (\log N_i)^{1/D} \geq 2$ is always met.
Since the iteration cannot go on indefinitely,
some set $A_i$ necessarily contains a
non-trivial solution to~\eqref{eq:intro:quadsyst},
and so does $A$ by translation and dilation invariance.
\qed

\section{Analytic preparation}
\label{sec:setup}

Our goal is now to
prove Proposition~\ref{thm:dincr:itlemma}.
For the rest of the article, we thus
fix an integer $s \geq 7$ and 
coefficients $\lambda_1,\dotsc,\lambda_s \in \Z \smallsetminus \{0\}$ 
such that $\lambda_1 + \dotsb + \lambda_s = 0$, and such that at
least two of the $\lambda_i$ are positive and at least two are negative.
We also fix an integer $N \geq 1$, but at this point we do not
impose size conditions on it, and we introduce
those as our argument progresses.
The same applies to the subset $A$ of $[N]$ of density $\delta$
which will be introduced later on,
although we fix the notation $f_A = 1_A - \delta 1_{[N]}$ here.

More importantly, from this point onwards, 
we let all further implicit or explicit constants
depend on $s$ and on $\lambda_1,\dotsc,\lambda_s$. 
While it would be possible in theory to track down 
all of these dependencies, it would require a sizeable effort
on our part, which we do not think worthwile in light of the asymptotic
nature of our results.
Furthermore, such a process would almost certainly not
allow us to eliminate the dependency of the logarithmic exponent 
in Theorem~\ref{thm:intro:mainthm}
on $s$ and on $\lambda_1,\dotsc,\lambda_s$,
given the shape of the density-increment statement
in the previous section.

Throughout the article we embed the interval $[N]$
in a larger interval $[M]$ for Fourier analytic purposes,
where $M$ is prime number
of magnitude $M \sim 2 ( |\lambda_1| + \dotsb + |\lambda_s| ) \cdot N$
chosen via Bertrand's postulate.
Therefore, we have $(M,\lambda_1 \dots \lambda_s) = 1$
and for integers $n_1,\dots,n_s \in [N]$, the system 
of equations~\eqref{eq:intro:quadsyst} is equivalent to
\begin{align*}
    \lambda_1 n_1 + \dotsb + \lambda_s n_s &\equiv 0 \pmod M, \\
    \lambda_1 n_1^2 + \dotsb + \lambda_s n_s^2 &\equiv 0 \pmod {M^2}.
\end{align*}
We use both the discrete and continuous (quadratic) Fourier transforms
in this article: the discrete transform is more convenient
for combinatorics, while the continuous one is more suitable
for number theory.
Precisely, given a function $f : \Z \rightarrow \C$ we define
$S_f : \Z_M \times \Z_{M^2} \rightarrow \C$
and $V_f : \T^2 \rightarrow \C$ by
\begin{align}
	\label{eq:setup:Ssum}
   	&&
    S_f(x,y)
    &= \E_{n \in [M]} f(n)
    e\Big( \frac{xn}{M} + \frac{yn^2}{M^2}\Big)
   	&&((x,y) \in \Z_M \times \Z_{M^2}),
    \\
    \label{eq:setup:Vsum}
    &&
    V_f(\alpha,\beta)
    &= \sum_{n \in \Z} f(n) e(\alpha n + \beta n^2).
    &&
    (\alpha,\beta \in \T).
\end{align}
In practice we usually work with functions $f : [N] \rightarrow \C$,
which we view throughout the article 
as functions on $\Z$ or $[M]$ with support in $[N]$.
We also define the normalized counting operator
acting on functions $f_1,\dots,f_s : \Z \rightarrow \C$ by
\begin{align}
\label{eq:setup:countingop}
    T(f_1,\dots,f_s) =
    M^{-(s-3)} \sum_{n_1,\dots,n_s \in \Z \,:\, \eqref{eq:intro:quadsyst}}
    f_1(n_1) \cdots f_s(n_s).
\end{align}
From the previous congruence considerations
and by discrete Fourier inversion, we 
deduce the following harmonic expression
for the counting operator.
\begin{proposition}
\label{thm:setup:countingopfourier}
    For functions $f_1,\dotsc,f_s : [N] \rightarrow \C$, we have
    \begin{align*}
       T(f_1,\dots,f_s) = \sum_{\mathbf{z} \in \Z_M \times \Z_{M^2}}
       S_{f_1}(\lambda_1 \bfz) \cdots S_{f_s}(\lambda_s \bfz)
    \end{align*}
\end{proposition}

\section{Discrete restriction estimates}
\label{sec:restr}

In this short section, we translate to the
discrete setting a restriction estimate
of Bourgain~\cite{bourgain:restrlattice}, 
which will prove extremely useful in the sequel.
We start with (the $d$-dimensional version of) a very useful lemma 
due to Marcinkiewicz and Zygmund,
and instrumental in Green's proof of 
Roth's theorem in the primes~\cite{green:rothprimes}.

\begin{proposition}
\label{thm:restr:discretize}
    Suppose that $M_1,\dots,M_d \geq 1$ are integers and let $p \geq 1$.
    We have, for every function
    $f : \Z^d \rightarrow \C$ with support in $[M_1] \times \dotsb \times [M_d]$,
    \begin{align*}
        \sum_{r_1 \in \Z_{M_1}} \cdots \sum_{r_d \in \Z_{M_d}}
        \Big| \wh{f}\Big( \frac{r_1}{M_1},\dotsc,\frac{r_d}{M_d} \Big)
        \Big|^p
        \ll M_1 \cdots M_d
        \idotsint\limits_{\T^d} | \wh{f}(\theta_1,\dotsc,\theta_d) |^p
        \mathrm{d}\theta_1 \dots \mathrm{d}\theta_d,
\end{align*}
where the implicit constant depends on $d$ and $p$ only.
\end{proposition}

\begin{proof}
Define the usual triangular functions
$\Delta_{M_i} : \Z \rightarrow \C$
by $\Delta_{M_i}(m) = (1 - \frac{|m|}{M_i})^+$
for $1 \leq i \leq d$,
and consider the tensor product
$\Delta_{\bfM} = \Delta_{M_1} \otimes \dotsb \otimes \Delta_{M_d}$.
We have $ \Delta_{M_i}(m_i) + 1 \leq 2 \Delta_{2M_i}(m_i)$
for all $m_i \in [1,M_i]$ for all $i$ (with equality in fact), 
and by taking products over $i \in [d]$ we can deduce that
$g_{\bfM} \coloneqq 2^d \Delta_{2 \bfM}
- \Delta_{\bfM}$ is at least $1$ on
$[M_1] \times \dotsc \times [M_d]$.
The proof of~\cite[Lemma~6.5]{green:rothprimes}
now generalizes straightforwardly to the $d$-dimensional setting
upon replacing the function $g$ from there by $g_{\bfM}$.
\end{proof}

The key restriction estimate we need is the following,
established by Bourgain~\cite{bourgain:restrlattice}
in the continous setting.

\begin{proposition}[Bourgain]
\label{thm:restr:restrbound}
    Let $p > 6$.
    We have, uniformly in functions $f : [M] \rightarrow \C$,
    \begin{align*}
        \| S_f \|_p \ll_p \| f \|_{L^2(M)}.
    \end{align*}
\end{proposition}

\begin{proof}
Apply Proposition~\ref{thm:restr:discretize}
to the function $g(n,m) = f(n) 1(m=n^2)$ supported on $[M] \times [M^2]$.
Recalling~\eqref{eq:setup:Ssum} and~\eqref{eq:setup:Vsum}, this yields
\begin{align*}
    \| S_f \|_p^p
    &= M^{-p}
        \sum_{x \in \Z_M} \sum_{y \in \Z_{M^2}}
        	\Big| \wh{g}\Big( \frac{x}{M}, \frac{y}{M^2} \Big) \Big|^p \\
    &\ll_p M^{-p} \cdot M \cdot M^2 
        \iint_{\T^2} 
        	|\wh{g}(\alpha,\beta)|^p
        \mathrm{d}\alpha \mathrm{d}\beta \\
    &= M^{3-p} \| V_f \|_p^p.
\end{align*}
The proposition then follows from Bourgain's restriction 
estimate~\cite[(3.115)]{bourgain:restrlattice}
(see~\cite[(1.7),~(3.1)]{bourgain:restrlattice} for
the notation used there)
and renormalizing.
\end{proof}

We note that using Keil's estimate~\cite[Theorem 2.1]{keil:diagonal} instead,
we could obtain an analogue of Proposition~\ref{thm:restr:restrbound}
with $\| f \|_\infty$ in place of $\| f \|_2$,
however this would not suffice for our argument in its present form.

\section{From non-uniformity to large energy}
\label{sec:multlin}

In this section we proceed with the first step
of our energy increment strategy,
which consists in converting the physical-space information 
that a set contains few solutions to~\eqref{eq:intro:quadsyst}
into an exploitable harmonic information: 
that the quadratic Fourier transform of this set
has a large $s$-th moment.
This part of our argument is very similar
to the work of Keil~\cite{keil:diagonal} and Smith~\cite{smith:diagonal}, 
and in particular we borrow highly non-trivial number-theoretic estimates from there;
one minor technical difference (at this point) is that
we work with the discrete Fourier transform.

We first need to bound the number of 
trivial solutions to~\eqref{eq:intro:quadsyst},
and for that purpose we import two supplementary
estimates on even moments of exponential sums from the 
litterature:~\cite[Proposition~2.1]{bourgain:restrlattice}
and~\cite[Lemma~5.1]{keil:diagonal}.
Both of these concern the exponential sum 
$V(\alpha,\beta) = \sum_{n \in [N]} e(\alpha n + \beta n^2)$
viewed as a function $V : \T^2 \rightarrow \C$,
and can be established by relatively
elementary divisor considerations.

\begin{proposition}
\label{thm:multlin:evenmoments}
    We have
    \begin{align*}
        \| V \|_4^4 \ll N^2,
        \quad\text{and}\quad
        \| V \|_6^6 \ll N^3 \log N.
    \end{align*}
\end{proposition}

We can now bound the number of trivial solutions easily,
assuming that $s \geq 7$.

\begin{proposition}
\label{thm:multlin:trivsols}
    Suppose that $A \subset [N]$ contains no non-trivial
    solutions to~\eqref{eq:intro:quadsyst}.
    Then
    \begin{align*}
        T(1_A,\dotsc,1_A) \ll \frac{\log N}{N}.
    \end{align*}
\end{proposition}

\begin{proof}
    Consider distinct indices $i,j \in [s]$.
    The number of solutions $(n_1,\dots,n_s) \in [N]^s$
    to~\eqref{eq:intro:quadsyst} 
    with $n_i = n_j$ is bounded by the number
    of solutions $(m_1,\dots,m_r) \in [N]^r$ 
    to a new system of the form
    \begin{align*}
        \mu_1 m_1 + \dots + \mu_r m_r &= 0, \\
        \mu_1 m_1^2 + \dots + \mu_r m_r^2 &= 0,
    \end{align*}
    where $\mu_i \in \Z \smallsetminus \{0\}$,
    and $r = s-2$ or $r=s-1$ 
    according to whether $\lambda_i + \lambda_j = 0$ or not.
    By the continuous circle method and Hölder's inequality,
    the number of such solutions is at most
    \begin{align*}
        \Big| \int_{\T^2} V( \mu_1 \bfgamma)
        \cdots V( \mu_r \bfgamma)
        \mathrm{d}\bfgamma \Big|
        \leq \| V \|_r^r,
    \end{align*}
    where we have used the $1$-periodicity of $V$.
    When $r \geq 6$ we have, by the second bound
    in Proposition~\ref{thm:multlin:evenmoments},
    \begin{align*}
        \| V \|_r^r \leq \| V \|_\infty^{r-6} \| V \|_6^6
        \ll N^{r-3} \log N \leq N^{s-4} \log N.
    \end{align*}
    Since $s \geq 7$, the only other possible case
    is $s=7$, $r=5$ for which the first bound
    in Proposition~\ref{thm:multlin:evenmoments}
    provides a similar (and even stronger) conclusion.
    Summing finally over all $\binom{s}{2}$ indices
    $i,j$ and recalling the normalizing factor
    in~\eqref{eq:setup:countingop},
    this concludes the proof.
\end{proof}

We also need an estimate on the number
of solutions to the system~\eqref{eq:intro:quadsyst}
in the complete integer interval $[N]$.
Luckily for us, this is provided by the delicate
circle method analysis of Keil~\cite[Proposition~7.1]{keil:diagonal}
and Smith~\cite{smith:diagonal}.

\begin{proposition}[Smith, Keil]
\label{thm:multlin:intgsols}
    For $N \geq C$, we have
    \begin{align*}
        T( 1_{[N]} , \dots , 1_{[N]} ) \gg 1.
    \end{align*}
\end{proposition}

It is now easy to derive the conclusion of large Fourier energy
in the non-uniform case, via the usual multilinear expansion process
and Hölder's inequality.

\begin{proposition}[Non-uniformity implies large energy]
\label{thm:multlin:largeenergy}
    Suppose that $A$ is a subset of $[N]$
    of density $\delta$ containing no non-trivial solutions 
    to~\eqref{eq:intro:quadsyst}, and that $N \geq C\delta^{-2s}$.
    Then
    \begin{align*}
        1 \ll \| S_{f_A / \delta} \|_s.
    \end{align*}
\end{proposition}

\begin{proof}
    We start from the bound of Proposition~\ref{thm:multlin:trivsols},
    and we expand $1_A = f_A + \delta 1_{[N]}$ by multilinearity to obtain
    \begin{align*}
        O(N^{-1/2}) 
        &= T(1_A,\dots,1_A)  \\
        &= \delta^s T( 1_{[N]}, \dotsc , 1_{[N]} )
            + \sum T( * , \dotsc , f_A , \dotsc , * ),
\end{align*}
where the sum has $2^s - 1$ terms
and the stars denote functions equal to $f_A$
or $\delta 1_{[N]}$.
By Proposition~\ref{thm:multlin:intgsols},
we know that $T( 1_{[N]}, \dotsc , 1_{[N]} ) \gg 1$,
and from our assumption on $N$ and 
the pigeonhole principle we may obtain a bound of the form
\begin{align*}
    \delta^s \ll |T(f_1,\dots,f_s)|,
\end{align*}
where $\ell \geq 1$ of the functions $f_i$
are equal to $f_A$ and $s-\ell$ of them
are equal to $\delta 1_{[N]}$.
By Proposition~\ref{thm:setup:countingopfourier}
and Hölder's inequality, and since $M$ is coprime to the $\lambda_i$, 
it follows that
\begin{align*}
    \delta^s 
    \ll \| S_{f_A} \|_s^{\ell}
    \cdot \delta^{s-\ell} \| S_{1_{[N]}} \|_{s}^{s-\ell}
    \ll \delta^{s-\ell} \| S_{f_A} \|_s^{\ell}
\end{align*}
where we have used the restriction bound
of Proposition~\ref{thm:restr:restrbound}
in the last inequality.
The inequality above is easily transformed
into the desired result.
\end{proof}

\section{Obtaining a large restricted energy}
\label{sec:energy}

We now have at our disposal an
$\ell^s$ estimate of the form
$\sum_{\bfz} |S_{f_A/\delta}(\bfz)|^s \gg 1$.
In the spirit of the Heath-Brown-Szemerédi energy-increment 
strategy~\cite{heath-brown:roth,szemeredi:roth},
our next move is to extract a larger restricted moment
$\sum_{i=1}^R |S_{f_A/\delta}(\bfz_i)|^r \gg R^c$ with $6 < r < s$,
where $R$ is a certain ``gain'' parameter of manageable size.
This is made possible by the following 
innocent-looking combinatorial lemma,
which is implicitely present in the exposition 
of Green~\cite[Lemma~4]{green:roth}.

\begin{lemma}[Energy pigeonholing]
\label{thm:energy:Lspigeon}
    Let $X \geq 1$ and $0 < r < s$ be parameters.
    Suppose that $(a_k)_{k \geq 1}$ 
    is a finite non-increasing sequence of 
    non-negative real numbers such that
    \begin{align}
    \label{eq:energy:sumsak}
        \sum_k a_k^s \gg 1
        \quad\text{and}\quad
        \sum_k a_k^r \leq X.
    \end{align}
    Then there exists $1 \leq R \ll_{r,s} X^{s/(s-r)}$
    such that
    \begin{align*}
        \sum_{k \leq R} a_k^r \gg_{r,s} R^{(s-r)/2s}.
    \end{align*}
\end{lemma}

\begin{proof}
Since $(a_k)$ is non-increasing,
we have a Markov-type bound
\begin{align*}
    j a_j^r \leq \sum_{k \leq j} a_k^r \leq X,
\end{align*}
that is $a_j \leq X^{1/r} j^{-1/r}$ for all $j$.
Since $s > r$ we have therefore, for any $Y \geq 1$,
\begin{align*}
    \sum_{k > Y} a_k^s
    \leq X^{s/r} \sum_{k > Y} k^{-s/r}
    \ll_{r,s} X^{s/r} Y^{-(s-r)/r}.
\end{align*}
Recalling~\eqref{eq:energy:sumsak} and
choosing $Y = C X^{s/(s-r)}$ with $C = C(r,s)$ large enough,
we have
\begin{align}
\label{eq:energy:minosumak}
    \sum_{k \leq Y} a_k^s \gg 1.
\end{align}
Let $\theta,\eta > 0$, and assume for contradiction
that $a_k \leq \theta k^{-(1+\eta)/s}$ for 
all $1 \leq k \leq Y$.
Then
\begin{align*}
    \sum_{k \leq Y} a_k^s
    \leq \theta \sum_{k \leq Y} k^{-1-\eta} 
    \ll_\eta \theta.
\end{align*}
Choosing $\theta = c(\eta)$ small enough, 
this contradicts~\eqref{eq:energy:minosumak}
and therefore we have found $1 \leq R \leq Y$
such that $a_R \gg_\eta R^{-(1+\eta)/s}$.
Hence, by monotonicity again,
\begin{align*}
    \sum_{k \leq R} a_k^r 
    \geq R a_R^r \gg_\eta R^{1 - (1+\eta) r/s}.
\end{align*}
Choosing $\eta = (s-r)/2r$,
we obtain the desired conclusion.
\end{proof}

\begin{corollary}
\label{thm:energy:restrenergy}
	Suppose that $A$ is a subset of $[N]$ of density $\delta$
	containing no non-trivial solutions 
	to~\eqref{eq:intro:quadsyst},
	and that $N \geq C\delta^{-2s}$.
	Then there exists $1 \leq R \ll \delta^{-C}$
	and distinct frequencies $\mathbf{z_1},\dotsc,\mathbf{z}_R
	\in \Z_M \times \Z_{M^2}$ such that
	\begin{align*}
		\sum_{i=1}^R |S_{f_A/\delta}(\bfz_i)|^{6.1} \gg R^c.
	\end{align*}1
\end{corollary}

\begin{proof}
We know that $\sum_{\bfz} |S_{f_A/\delta}(\bfz)|^s \gg 1$
by Proposition~\ref{thm:multlin:largeenergy},
and that $\sum_{\bfz} |S_{f_A/\delta}(\bfz)|^r \ll_r \delta^{-r/2}$
for every $r > 6$ by Proposition~\ref{thm:restr:restrbound}.
We fix $r = 6.1 < 7 \leq s$ for definiteness.
Ordering the absolute values $|S_{f_A/\delta}(\bfz)|$ 
for $\bfz\in \Z_M \times \Z_{M^2}$ by size,
and applying Lemma~\ref{thm:energy:Lspigeon}, 
we obtain the desired conclusion.
\end{proof}

\section{Linearization and density increment}
\label{sec:lin}

In this section we carry out the most technical
part of our argument, by which we turn the previous
large restricted energy into a density increment.
We have previously isolated a collection
of quadratic harmonics $| S_{f_A/\delta}(\bfz_i)|$
which is large in an $\ell^{6.1}$ sense,
and the next three steps consist in replacing the function $f_A$ 
in each of these harmonics by its convolution $\wt{f}_A$ 
with smaller progressions, extracting a large
second moment of $\wt{f}_A$ via a restriction bound, 
and carrying out the classical $L^2$ density increment strategy.
To alleviate technical statements, 
we fix an integer $R$ and a collection of
frequencies $\bfz_1,\dots,\bfz_R$ throughout this section,
and we define quadratic polynomials $\phi_i(n) = \alpha_i n + \beta_i n^2$,
where $\bfz_i = (x_i,y_i)$ and $\alpha_i = x_i/M, \beta_i = y_i/M^2$.
Recalling~\eqref{eq:setup:Ssum} we thus have,
for every function $h : \Z \rightarrow \C$ and $i \in [R]$,
\begin{align}
\label{eq:lin:quadratics}
	S_h(\bfz_i) = \E_{n \in [M]} h(n) e( \phi_i(n) ).
\end{align}

The only non-trivial fact from diophantine approximation that we require 
is a version by Green and Tao~\cite[Proposition~A.2]{GT:quantsz} of Schmidt's
result on simultaneous quadratic recurrence, where the
dependency of the error in $d$ is made explicit.

\begin{proposition}[Schmidt, Green-Tao]
\label{thm:lin:quadrec}
\label{thm}
    Let $\theta_1,\dots,\theta_d \in \T$ be real numbers modulo $1$,
    and let $X \geq 1$ be an integer.
    Then there exists an integer $1 \leq q \leq X$ such that
    \begin{align*}
        &\phantom{\text{for all}\quad 1 \leq i \leq d} &
        \| q^2 \theta_i \| &\ll d X^{-c/d^2}
        &&\text{uniformly in}\quad 1 \leq i \leq d.
    \end{align*}
\end{proposition}

We note here that for the purpose of proving
Theorem~\ref{thm:intro:mainthm}, we could work
equally well with the slightly weaker estimate
of Croot, Lyall and Rice~\cite[Theorem 1]{CLR:quadrec},
which has the advantage of having a completely
elementary proof.
Yet we stick with Proposition~\ref{thm:lin:quadrec},
only because the error term there is more concise to state.
Our first step, then, is to find a collection
of progressions on which the quadratics phases 
in~\eqref{eq:lin:quadratics} are nearly constant.
In this we follow Green and Tao~\cite{GT:quantsz},
using however the language of translates rather than
that of partitions.

\begin{proposition}[Simultaneous linearization of quadratics]
\label{thm:lin:simultlin}
    Let $\eps,\delta \in (0,1]$ be parameters,
    and suppose that $N \geq (2/\eps\delta)^{CR^3}$.
    Then there exist integers $q,(r_n)_{n \in \Z},U,V$ of size
    \begin{align*}
        &&
        U &\sim N^{c/R^2}    &     V&\sim U^{c/R}
        && \\
        &&
        q &\leq N^{1/4}      &     r_n &\leq U^{1/4}
        &&
    \end{align*}
    such that for $P = q[U]$, $Q_n = q r_n [V]$, we have
    \begin{align*}
        \| \phi_i( n + m + k ) - \phi_i( n + m ) \| \leq \eps\delta
    \end{align*}
    for all $n \in [N], m \in P, k \in Q_n$ and $i \in [R]$.
\end{proposition}

\begin{proof}
    We consider at first arbitrary integers $n,q,a,b$,
    and we write
    \begin{align}
    \label{eq:lin:phidiff}
        \begin{split}
        &\phantom{=}\ \phi_i(n + qa) - \phi_i(n + qb) \\
        &= \alpha_i \big[ (n + qa) - (n + qb) \big]
        + \beta_i \big[  (n+qa)^2 - (n+qb)^2 \big] \\
        &= (a^2 - b^2) q^2 \beta_i + (a-b)\gamma_{i,n},
        \end{split}
    \end{align}
    where $\gamma_{i,n} = q(\alpha_i + 2n\beta_i)$.
    Via Proposition~\ref{thm:lin:quadrec},
    we now choose $1 \leq q \leq N^{1/4}$  such that
    $\| q^2 \beta_i \| \ll R N^{-c_0/R^2}$ for all $i \in [R]$.
    We also pick an integer $U \sim N^{c_0/4R^2}$
    and we let $a = x + r_n y$ and $b = x$,
    where $x \in [U]$ and $y \in \Z$ is arbitrary, 
    and where $1 \leq r_n \leq U^{1/4}$ are chosen so that
    $\| r_n \gamma_{i,n} \| \leq U^{-c_1/R}$ for all $i \in [R]$ and $n \in \Z$,
    via Dirichlet's theorem on simultaneous linear
    recurrence\cite[Theorem II.1A]{schmidt:diophapprox}.
    If we now pick another integer $V \sim U^{c_1/2R}$ and
    insist that $y,z \in [V]$, it follows
    from~\eqref{eq:lin:phidiff} that
    \begin{align*}
        \| \phi_i(n + qa) - \phi_i(n + qb) \|
        &\leq |a^2-b^2| \cdot \| q^2 \beta_i \|
            + |y-z| \cdot \| r_n \gamma_{i,n} \| \\
        &\ll R U^2 N^{-c_0/R^2} + V U^{-c_1/R} \\
        &\ll R N^{-c_0 / 2R^2} + N^{- c_2 / R^3}
    \end{align*}
    with $c_2 = c_0 c_1 / 8$,
    and the error is less than $\eps\delta$ for $N \geq (2/\eps\delta)^{CR^3}$.
\end{proof}

We chose the sizes of the parameters
in the previous proposition
so that $P \subset [N^{1/2}]$ and $r_n [V] \subset [U^{1/2}]$.
Before proceeding further, we 
recall two standard averaging techniques,
which we use implicitely throughout the section.

\begin{proposition}[Regularity calculus]
\label{thm:lin:reg}
    Let $\eps \in (0,1]$ be a parameter
    and let $f : \Z \rightarrow \C$ be a function.
    Consider two arithmetic progressions 
    $P = q [N]$ and $P' \subset q [N']$,
    where $q,N,N' \geq 1$.
    Then
    \begin{align*}
        &\textrm{\upshape [Shifting]} &
        \E_{n \in P} f(n + n') 
        &= \E_{n \in P} f(n) + O\big( \tfrac{N'}{N} \| f \|_\infty \big)	
        &&\forall\ n' \in P', \\
        &\text{\upshape [Sub-averaging]} &
        \E_{n \in P,\, n' \in P'} f(n + n')
        &= \E_{n \in P} f(n) + O\big( \tfrac{N'}{N} \| f \|_\infty \big).
        &&
\end{align*}
\end{proposition}

The next proposition allows us to approximate
the quadratic Fourier transform of a function 
by the transform of an additively 
smoothed version of itself,
at several frequencies at once.
Since we have the function $f_A/\delta$ in mind
we assume an $L^\infty$ bound of the form $\delta^{-1}$ below.
We also adopt a handy notation:
for complex numbers $X,Y$
we write $X \approx_\eps Y$ when $|X-Y| \ll \eps$.

\begin{proposition}[Additive smoothing]
\label{thm:lin:smoothing}
    Let $\eps,\delta \in (0,1]$ be parameters,
    and assume that $N \geq (2/\eps\delta)^{CR^3}$ and
    $P,(Q_n)_{n \in \Z}$ are as in Proposition~\ref{thm:lin:simultlin}.
    Given a function $h : [N] \rightarrow \C$, define 
	$\wt{h} : [N] \rightarrow \C$ by
    \begin{align}
    \label{eq:lin:smoothing}
        \wt{h}(n) = \E_{m \in P,\, k \in Q_{n-m}} h(n + k).
    \end{align}
    Provided that $\| h \|_\infty \leq \delta^{-1}$,
    we then have
    \begin{align}
    	\label{eq:lin:phasesmoothing}
        \E_{n \in [N]} h(n)e(\phi_i(n))
        &\approx_\eps
        \E_{n \in [N]} \wt{h}(n)e(\phi_i(n))
		\qquad
		\text{for all $i \in [R]$,}
        \\
        \label{eq:lin:avgsmoothing}
        \E_{n \in [N]} h(n)
        &\approx_\eps
        \E_{n \in [N]} \wt{h}(n).
    \end{align}
\end{proposition}

\begin{proof}
Sub-averaging twice, 
using the bound $N \geq (\eps\delta)^{-CR^2}$, we derive
\begin{align*}
	I &\coloneqq
	\E_{n \in [N]} h(n) e( \phi_i(n) ) \\
	&\approx_{\eps} \E_{n \in [N]} \E_{m \in P} \E_{k \in Q_n}
	h(n + m + k) e( \phi_i(n + m + k) ).
\end{align*}
By Proposition~\ref{thm:lin:simultlin}, we thus have
\begin{align*}
	I &\approx_{\eps} \E_{n \in [N],\, m \in P,\, k \in Q_n}
	h(n + m + k) e( \phi_i (n+m) ) \\
	&= \E_{m \in P} \E_{n \in [N]} \Big[ \E_{k \in Q_{(n+m) - m}}
				h( (n + m) + k ) e( \phi_i(n + m) ) \Big].
\end{align*}
By shifting in $n$ at fixed $m$, using the bound $N \geq (\eps\delta)^{-2}$, we obtain
\begin{align*}
	I  \approx_\eps \E_{m \in P} \E_{n \in [N]} \E_{k \in Q_{n-m}}
	h( n + k ) e( \phi_i(n) )
\end{align*}
After reordering variables, we have 
$I \approx_\eps \E_{n \in [N]} \wt{h}(n) e( \phi_i(n) )$
as expected, and an identical computation gives
an analog approximation where $e(\phi_i(n))$ is replaced by~$1$.
\end{proof}

We can now perform the promised transformation
of the function $f_A/\delta$ in Corollary~\ref{thm:energy:restrenergy} into its smoothed
version, which is then guaranteed to have a large second moment via
the $\ell^s \,\text{--}\, L^2$ restriction estimate of Section~\ref{sec:restr}.

\begin{proposition}
\label{thm:lin:largeL2}
	Let $\eps,\delta \in (0,1]$ be parameters,
	and assume that $N \geq (2/\eps\delta)^{CR^3}$
	and $P,(Q_n)_{n \in \Z}$ are as in Proposition~\ref{thm:lin:simultlin}.
	Define $h \mapsto \wt{h}$ as in~\eqref{eq:lin:smoothing}
	Suppose also that
	$A$ is a subset of $[N]$ of density $\delta$ such that
	\begin{align}
	\label{eq:lin:largeenergy}
		R^c \ll \sum_{i=1}^R |S_{f_A / \delta}(\bfz_i)|^{6.1}.
	\end{align}
	Provided that $\eps \leq c/R$, we then have
	\begin{align*}
		R^c \ll \bigg\| \frac{\wt{f_A}}{\delta} \bigg\|_{L^2(N)}^2.
	\end{align*}
\end{proposition}

\begin{proof}
Since $M \sim CN$ and the functions under consideration
are supported in $[N]$, we can pass from averages
over $[M]$ to averages over $[N]$ and back, 
losing only a constant factor in the process.
We rescale the averages~\eqref{eq:lin:quadratics} to $[N]$
and apply~\eqref{eq:lin:phasesmoothing} to $h = f_A/\delta$,
so that by the inequality $|x+y|^{6.1} \ll |x|^{6.1} + |y|^{6.1}$ we have
\begin{align*}
	c R^c - C R \eps^{6.1} \leq 
	\sum_{i = 1}^R \bigg| \E_{n \in [N]} \frac{\wt{f_A}}{\delta}(n) e( \phi_i(n) ) \bigg|^{6.1}.
\end{align*}
Assuming that $\eps \leq c/R$ with $c$ small enough (say), 
the left-hand side is $\gg 1$.
Rescaling the average on the right-hand side to $[M]$
and completing the sum, we can apply the restriction estimate
from Proposition~\ref{thm:restr:restrbound}
and rescale back to $[N]$ to finish the proof.
\end{proof}

We know from the last proposition that
the balanced function of our set, averaged over 
a family of small progressions,
has a large $L^2$ norm.
One can pass from this information to a density increment
by a standard argument combining regularity computations 
and an $L^{\infty}-L^{1}$ bound,
much as in the Bohr set setting of 
Roth's theorem~\cite[Proposition~2]{me:bourgainroth}.
At this stage we rescale norms to the interval $[N]$
on which our functions naturally live.

\begin{proposition}[$L^2$ density increment]
\label{thm:lin:L2dincr}
	Let $\eps,\delta \in (0,1]$ and $\nu > 0$ be parameters,
	and suppose that ${N \geq (2/\eps\delta)^{CR^3}}$ and
	$P,(Q_n)_{n \in \Z}$ 
	are as in Proposition~\ref{thm:lin:simultlin},
	$h \mapsto \wt{h}$ is defined by~\eqref{eq:lin:smoothing},
	and $A$ is a subset of $[N]$ of density $\delta$ such that
	\begin{align*}
		\bigg\| \frac{\wt{f_A}}{\delta} \bigg\|_{L^2(N)}^2 \geq \nu.
	\end{align*}
	If $\eps \leq c\nu\delta$,
	there exists an arithmetic progression
	$Q$ contained in $[N]$ such that
	\begin{align*}
		|A \cap Q| / |Q| &\geq (1 + \nu/2) \cdot \delta, \\
		|Q| &\geq N^{c/R^3}.
	\end{align*}
\end{proposition}

\begin{proof}
We write $\| h \|_2 = \| h \|_{L^2(N)}$
for functions $h : [N] \rightarrow \C$ for conciseness.
Before engaging in computations notice that,
by~\eqref{eq:lin:smoothing} and considering the
lengths of the arithmetic progressions $(Q_{n})_{n \in \Z}$,
we have, for every $n \in [N]$,
\begin{align}
\label{eq:lin:mollifindic}
\wt{1_{[N]}} (n) =
	\begin{cases}
		= 1 & \text{if $n \in [ 1 , N - N^{1/2} [$}, \\
		\in [0,1] & \text{always}.
	\end{cases}
\end{align}
Therefore $\wt{1_{[N]}}$
acts as a mollified indicator function of $[N]$.
Recalling next that $f_A = 1_A - \delta 1_{[N]}$,
we can unfold the $L^2$ norm in the proposition to obtain
\begin{align}
\label{eq:lin:L2unfolding}
	\nu \leq \bigg\| \frac{\wt{1_A}}{\delta} \bigg\|_2^2
	- 2 \bigg\langle \frac{\wt{1_A}}{\delta} , \wt{1_{[N]}} \bigg\rangle
	+ \| \wt{1_{[N]}} \|_2^2.
\end{align}
A quick computation using~\eqref{eq:lin:mollifindic}
and~\eqref{eq:lin:avgsmoothing} yields
\begin{align}
\label{eq:lin:L2midterms}
	\bigg\langle \frac{\wt{1_A}}{\delta} , \wt{1_{[N]}} \bigg\rangle 
	\approx_\eps 1,
	\qquad
	 \| \wt{1_{[N]}} \|_2^2 \approx_\eps 1.
\end{align}
We want to avoid certain overwrapping scenarios
near the edges of our interval, 
and to do so we introduce the auxiliary
set $E = [1,N - N^{1/2}[$.
Writing $\mu_A = 1_A / \delta$ and using the
bound $N \geq \eps^{-2}\delta^{-4}$, we then have
\begin{align*}
	\bigg\| \frac{\wt{1_A}}{\delta} \bigg\|_2^2
	&= \E_{n \in [N]} 1_E(n) \wt{\mu_A}(n)^2 + O(\eps) \\
	&\leq \| 1_E \wt{\mu_A} \|_\infty \| \wt{\mu_A} \|_1 + O(\eps).
\end{align*}
Applying~\eqref{eq:lin:avgsmoothing} to $h = \mu_A$, we deduce that
\begin{align}
	\notag
	\bigg\| \frac{\wt{1_A}}{\delta} \bigg\|_2^2
	&\leq \| 1_E \wt{\mu_A} \|_\infty \| \mu_A \|_1 + O(\eps/\delta) \\
	\label{eq:lin:L2mainterm}
	&= \| 1_E \wt{\mu_A} \|_\infty + O(\eps/\delta).
\end{align}
Inserting~\eqref{eq:lin:L2midterms}
and~\eqref{eq:lin:L2mainterm} in~\eqref{eq:lin:L2unfolding},
we obtain
\begin{align*}
	\nu \leq \| 1_E \wt{\mu_A} \|_\infty - 1 + O( \eps / \delta ).
\end{align*}
Assuming that $\eps \leq c\nu\delta$,
it follows that there exists $n \in E$ such that
\begin{align*}
	1 + \nu/2 \leq \E_{m \in P} \E_{ k \in Q_{n-m}} \mu_A (n + k).
\end{align*} 
By the pigeonhole principle, we may therefore find
$m \in P$ such that
\begin{align*}
	(1 + \nu/2) \cdot \delta 
	\leq \E_{k \in Q_{n-m}} 1_A (n + k) 
	= \frac{|A \cap ( n + Q_{n-m})|}{|n + Q_{n-m}|}.
\end{align*}
Given our choice of $n \in E$ and the 
bounds on the size of parameters 
in Proposition~\ref{thm:lin:simultlin},
we are guaranteed that $n + Q_{n-m}$
is contained in $[N]$, and this concludes the proof.
\end{proof}

\section{Assembling all the pieces}
\label{sec:pieces}

In this section we finish the proof
of Proposition~\ref{thm:dincr:itlemma},
gathering the main statements of the two previous sections
and assigning explicit values to parameters.
As explained in Section~\ref{sec:dincr},
this completes the proof of Theorem~\ref{thm:intro:mainthm}.

\textit{Proof of Proposition~\ref{thm:dincr:itlemma}.}
By Corollary~\ref{thm:energy:restrenergy},
we can find a parameter $1 \leq R \ll \delta^{-C}$
and distinct frequencies $\bfz_1,\dotsc,\bfz_R$ such that
\begin{align*}
	R^c \ll \sum_{i=1}^R |S_{f_A/\delta}(\bfz_i)|^{6.1}.
\end{align*}
We assume that these frequencies
are those that we started with
at the beginning of Section~\ref{sec:lin},
and we fix $\eps = c\delta/R$.
Proposition~\ref{thm:dincr:itlemma} then follows
at once from Propositions~\ref{thm:lin:largeL2}
and~\ref{thm:lin:L2dincr}.
\qed

\section{Final remarks}
\label{sec:rks}

We make a few last technical comments on our method
in this section.
Consider a subset $A$ of $[N]$ of density $\delta$
containing no non-trivial solutions
to~\eqref{eq:intro:quadsyst}.
If we could draw a conclusion of the form
$R^c \ll \sum_{i = 1}^R |S_{1_A/\delta}(\bfz_i)|^s$
for distinct \textit{non-zero} frequencies $\bfz_1,\dotsc,\bfz_R$,
then we could use Keil's~\cite{keil:diagonal} restriction estimate
$\|S_f\|_s \ll \| f \|_\infty$ 
to obtain a slight simplification in our argument.
Indeed we could then deduce a density increment from
the lower bound $\| \wt{1_A} / \delta \|_{\infty} \geq 1 + cR^c$
(with the notation of Section~\ref{sec:lin}),
without resorting to any regularity computations.
However we do not see how to do this at present,
and the information $\| \wt{f_A} / \delta \|_\infty \gg R^c$
is hard to exploit by itself:
it provides either a density increment or
a density decrement, without the means to
decide between the two.

Finally, we close this article with a prediction
inspired by a remark of Smith~\cite[p.~276]{smith:diagonal}.
We remark that the density increment strategy employed
here can be adapted to handle any translation/dilation-invariant 
system of equations of the form
\begin{align}
\label{eq:rks:waringsyst}
\begin{split}
	\lambda_1 n_1 + \dotsb + \lambda_s n_s &= 0 \\
	&\vdots \\
	\lambda_1 n_1^k + \dotsb + \lambda_s n_s^k &= 0,
\end{split}
\end{align}
where $k \geq 1$ and $\lambda_1 + \dotsb + \lambda_s = 0$,
under two additional conditions:
that the classical circle method succeeds in
bounding by below the number of solutions 
to~\eqref{eq:rks:waringsyst} in $[N]$,
and that the discrete exponential sums
\begin{align*}
	W_f(x_1,\dotsc,x_k) = 
	\E_{n \in [M]} f(n) e\Big( \frac{x_1 n}{M} + \dotsb + \frac{x_k n^k}{M^k} \Big)
\end{align*} 
again satisfy a restriction estimate of the form
\begin{align*}
	\| W_f \|_s \ll_s \| f \|_{L^2(M)}.
\end{align*}
It may well be that both these conditions are met
for $s$ large enough with respect to $k$.

\bibliographystyle{amsplain}
\bibliography{logkeil_arxiv3}

\bigskip

\textsc{\footnotesize Department of mathematics,
University of British Columbia,
Room 121, 1984 Mathematics Road,
Vancouver BC V6T 1Z2, Canada
}

\textit{\small Email address: }\texttt{\small khenriot@math.ubc.ca}

\end{document}